\documentclass[11pt, a4paper]{amsart}
\usepackage{amsmath, amsthm, amssymb}

\usepackage[latin1]{inputenc}
\usepackage[english]{babel}
\usepackage{csquotes}

\usepackage{inconsolata}

\usepackage{colonequals}

\usepackage[backend=biber, datamodel=mrnumber, maxbibnames=99, sortcites]{biblatex}
\usepackage{hyperref}
\hypersetup{
    colorlinks=true,
    linkcolor=blue,
    citecolor=green,
    urlcolor=blue,
}
\usepackage[capitalise]{cleveref}

\usepackage{tikz, tikz-3dplot}
\usetikzlibrary{matrix, arrows, calc, cd}
\usepackage[twoside=false]{geometry}

\usepackage{thmtools}

\declaretheoremstyle[
  spaceabove = 3pt,
  spacebelow = 3pt,
  bodyfont = \itshape,
]{first}
\declaretheoremstyle[
  spaceabove = 3pt,
  spacebelow = 3pt,
]{second}
\declaretheorem[numberwithin=section, style=first]{theorem}

\declaretheorem[sibling=theorem, style=first]{corollary}
\declaretheorem[sibling=theorem, style=first]{lemma}
\declaretheorem[sibling=theorem, style=first]{proposition}

\declaretheorem[sibling=theorem, style=second]{example}
\declaretheorem[sibling=theorem, style=second]{remark}
\declaretheorem[sibling=theorem, style=second]{definition}

\Crefname{thm}{Theorem}{Theorems}
\Crefname{conj}{Conjecture}{Conjectures}
\Crefname{cor}{Corollary}{Corollaries}
\Crefname{lem}{Lemma}{Lemmas}
\Crefname{ex}{Example}{Examples}
\Crefname{rem}{Remark}{Remarks}
\Crefname{defn}{Definition}{Definitions}

\DeclareMathOperator{\Hom}{Hom}
\DeclareMathOperator{\GL}{GL}
\DeclareMathOperator{\id}{id}
\DeclareMathOperator{\sst}{sst}

\DeclareDocumentCommand\repspace{omo}{
	\IfNoValueTF{#1}{
		\mathrm{R}_{#2}
	}
	{
		\mathrm{R}^{#1}_{#2}
	}
	\IfValueTF{#3}{
		(#3)
	}
	{}
}

\newcommand\group[1]{\mathrm{G}_{#1}}
\newcommand\pgroup[1]{\mathrm{PG}_{#1}}

\mathchardef\mhyphen="2D
\newcommand\semistable[1]{#1\mhyphen\mathrm{sst}}
\newcommand\stable[1]{#1\mhyphen\mathrm{st}}
\newcommand\stabSpace[1]{{}^\bot{#1}}
\newcommand\sstCone[1]{\sst(#1)}
\newcommand\wall[1]{W_{#1}}

\newcommand{\Z}{\mathbb{Z}}
\newcommand{\Q}{\mathbb{Q}}
\newcommand{\R}{\mathbb{R}}

\newcommand{\bd}{\mathbf{d}}
\newcommand{\be}{\mathbf{e}}

\renewcommand{\bf}{\mathbf{f}}

\newcommand{\software}[1]{{\small{\sc#1}}}

\title{Finding the walls for quiver moduli}
\author{Hans Franzen}
\author{Gianni Petrella}
\author{Rachel Webb}
\date{\today}

\begin{document}

\begin{abstract}
We give an effective characterisation of the walls
in the variation of geometric invariant theory problem
associated to a quiver and a dimension vector.
\end{abstract}
\maketitle

\section{Introduction}
A quiver $Q$ with vertex set $Q_0$ and a dimension vector $\bd \in \mathbb{Z}^{Q_0}_{\geq 0}$
determine a representation $\repspace{\mathbf{d}}$ of a certain reductive group $G_{\bd}$.
A further choice of stability parameter $\theta \in \mathbb{R}^{Q_0}$
orthogonal to $\bd$ determines open loci of stable and semistable points
\begin{equation}
\repspace[\stable{\theta}]{\mathbf{d}} \subseteq \repspace[\semistable{\theta}]{\mathbf{d}} \subseteq\repspace{\mathbf{d}}.
\end{equation}
Let $\sstCone{\bd}$ denote the locus in the orthogonal complement $\stabSpace{\bd}$
of $\bd$ for which $\repspace[\semistable{\theta}]{\mathbf{d}}$ is nonempty.
We compute the locus of \textit{walls} in $\sstCone{\bd}$, namely the locus of $\theta$ such that
$\repspace[\stable{\theta}]{\mathbf{d}} \subsetneq \repspace[\semistable{\theta}]{\mathbf{d}}$.
Our main result is the following (see \cref{c:artin} and \cref{lem:walls}).
\begin{theorem}\label{t:main-intro}
The walls in $\sstCone{\bd}$ are a union of sets of the form $W_\be$,
where $\be \in \mathbb{Z}^{Q_0}_{\geq 0}$ is a nonzero vector such that $\be \neq \bd$ and $\be_i \leq \bd_i$ for each $i \in Q_0$, and
\begin{equation}
W_\be = \sstCone{\be} \cap \sstCone{\bd - \be}.
\end{equation}
\end{theorem}
\noindent
In \cref{t:semistable_generic} we recall an existing algorithm,
based on~\cite{Schofield:92} and \cite[Theorem~3]{DW:00}, for computing $\sstCone{\bd}$.
This gives an effective method for computing the walls.

Our result has several applications.
The first is to the problem of variation of geometric invariant theory (GIT) associated to $Q$ and $\bd$: in this setting,
two stability parameters~$\theta,\eta$ in the orthogonal complement of $\bd$
are \textit{GIT equivalent} if their loci of semistable points are equal.
\cref{t:main-intro} yields an effective algorithm for determining when $\eta$ and $\theta$ are GIT equivalent (see \cref{C:GIT}).
\begin{theorem}\label{t:GIT-intro}
The stability parameters $\theta, \eta$ are GIT equivalent if and only if, for every $\be \leq \bd$,
the convex hull of $\theta$ and $\eta$
is either contained in $W_\be$,
or does not meet $W_\be$.
\end{theorem}
\noindent
The set of GIT equivalence classes determines a fan, called the GIT fan of $(Q, \bd)$.
This is a fan in the orthogonal complement to $\bd$ with support equal to $\sstCone{\bd}$.
In \cref{C:fan} we give an effective algorithm for computing this fan.

For the second application, motivated by physics literature,
we say that the pair $(Q, \bd)$ \textit{has a geometric phase}
if there exists a choice of $\theta$
for which~$\repspace[\semistable{\theta}]{\mathbf{d}}=\repspace[\stable{\theta}]{\mathbf{d}}$
and this set is nonempty.
The following also appears as \cref{T:phase}.
\begin{theorem}
\label{alphathm:geometric-phase}
The pair $(Q, \bd)$ has a geometric phase
if and only if $\bd$ is an indivisible Schur root of $Q$, or equivalently
if and only if $\sstCone{\bd}$ is full-dimensional and all walls $W_\be$ have codimension one.
\end{theorem}

The third application is to better understand the \textit{wall system}
and \textit{special subdimension vectors} introduced in \cite[Section~3]{HP:02}
to study the GIT fan of quivers without oriented cycles.
The wall system is the smallest set of hyperplanes that contain those $W_\be$
of positive codimension in $\sstCone{\bd}$, and it is shown in \cite[Lemma~3.3]{HP:02}
that every hyperplane in the wall system is orthogonal to a special subdimension vector.
We give a recursive characterisation
of the special subdimension vectors
of $(Q, \bd)$ in \cref{L:special}.
Its use in \cref{ex:walls-strictly-contained-hyperplanes}
implies the following result.
\begin{theorem}
Even for acyclic quivers,
not every special subdimension vector
is orthogonal to a hyperplane in the wall system,
and not every hyperplane in the wall system
is contained in the union
of non-maximal cones in the GIT fan.
\end{theorem}

We provide implementations
of our algorithms for computing the walls $W_\be$ and the GIT fan in \cite{walls-and-chambers-implementation}.
These use the \software{QuiverTools} package
for \software{SageMath} and \software{Julia}, see \cite{quivertools,quivertools-paper}.
In \cref{sec:examples} we use these implementations to compute various examples,
including an illustration of how the GIT fan changes under quiver mutations
and the complete GIT fan for a realization of the Segre cubic as a quiver moduli space.

\begin{remark}
In \cite[p.163]{AH:09} the \textit{walls} are defined to be the cones in the GIT fan of codimension 1.
By contrast, our sets $W_\be$ can have any codimension in $\sstCone{\bd}$, including codimension $0$;
we show in \cref{c:walls} that the $W_\be$ of codimension 1
are precisely the walls of \cite{AH:09}.
We call the set $W_\be$ a ``wall'', even when it has codimension $0$,
because it always arises from a codimension $1$ wall of the abelianized quiver.
\end{remark}

\subsection*{Relation to other work}
For quivers without oriented cycles,
a description of the GIT fan is given in~\cite{Chindris:08} (but see also~\cite{CG:19,IPT:15}).
Our description of the walls (\cref{lem:walls}) can be derived
from the description of GIT cones that appears in op. cit.,
but that derivation is not simpler than what we present here.

Our discussion of GIT equivalence relies heavily
on the general structure of the GIT fan worked out in~\cite{AH:09}.

Large portions of \cref{alphathm:geometric-phase} follow from~\cite[Propositions~4.4 and~5.3]{King:94},
so we suspect that this result may be known to experts.
We nevertheless include the statement and a proof
as we could not find it in the literature.

Finally, we note that the GIT fan of $(Q, \bd)$ when $Q$ is acyclic and $\bd = (1, 1, \ldots, 1)$ is also studied in \cite{Hille:03}, and that an algorithm for computing the fan in this case has been implemented in \software{Macaualy2} (see \cite{BP:21}).

\subsection*{Conventions}
Matching the assumptions in \cite{AH:09},
the representation $\repspace{\bd}$ and group $G_\bd$
are defined
over an algebraically closed field $k$ of characteristic zero.
If $V$ is a real vector space and $U \subseteq V$ is a subset
invariant under the scaling action of $\mathbb{R}$,
the \textit{dimension} of $U$ is the dimension of its linear span.

\subsection*{Acknowledgements}
The authors thank Pieter Belmans, Harm Derksen, and Markus Reineke for helpful conversations.
H.F.~was partially supported by the Deutsche For\-schungsgemeinschaft (DFG, German Research Foundation)
SFB-TRR~358/1~2023 ``Integral Structures in Geometry and Representation Theory'' (491392403).
G.P. was supported by the Luxembourg National Research Fund (FNR-17953441).

\section{Background}\label{s:setup}

Let $Q$ be a quiver with vertex set $Q_0$ and arrow set $Q_1$,
and let $s, t: Q_1 \to Q_0$ denote the source and target maps.
For a dimension vector $\bd \in \mathbb{Z}^{Q_0}_{\geq 0}$, define
$\repspace{\mathbf{d}}$ and $\group{\mathbf{d}}$ respectively as
\begin{equation}\notag
\repspace{\mathbf{d}} = \bigoplus_{a \in Q_1} \Hom(k^{\bd_{s(a)}},k^{\bd_{t(a)}}),
\quad \quad \quad \quad \quad \quad
\group{\mathbf{d}} = \prod_{i \in Q_0} \GL(k^{\bd_i}).
\end{equation}

We call the elements of $\repspace{\mathbf{d}}$ \emph{representations of $Q$ of dimension vector $\bd$}.
The algebraic group $\group{\mathbf{d}}$ acts linearly on the left on~$\repspace{\mathbf{d}}$,
with the action of $g = (g_i) \in \group{\mathbf{d}}$ on the representation~$M=(M_a)\in\repspace{\mathbf{d}}$
defined as
\begin{equation}
g\cdot M \colonequals (g_{t(a)}M_{a}g_{s(a)}^{-1}).
\end{equation}
The central closed subgroup $\Delta = \{z\cdot\id \mid z \in k^\times\}$ of $\group{\mathbf{d}}$ acts trivially,
so we have an induced action
of $\pgroup{\mathbf{d}}\colonequals \group{\mathbf{d}}/\Delta$ on $\repspace{\mathbf{d}}$.
Two points $M, N \in \repspace{\mathbf{d}}$ are \emph{isomorphic as representations}
if and only if they lie in the same $\group{\bd}$-orbit.

A real vector $\theta \in \R^{Q_0}$ defines a function on the set $\smash{\Z^{Q_0}_{\geq 0}}$ of dimension vectors
by setting $\theta(\bd)\colonequals\theta \cdot \bd$, where $\bd$ is an element of $\smash{\Z^{Q_0}_{\geq 0}}$ and $\theta \cdot \bd$ is the dot product.
If $M$ is a representation of $Q$ of dimension vector $\be$
(that is, $M\in\repspace{\be}$),
we define $\theta(M)\colonequals\theta(\be)$.

Given a dimension vector $\bd \in \smash{\mathbb{Z}^{Q_0}_{\geq 0}}$, the associated \textit{space of stability parameters} is
\begin{equation}
\stabSpace{\mathbf{d}} \colonequals \{\theta \in \R^{Q_0} \mid \theta( \mathbf{d}) = 0\}.
\end{equation}
A choice of stability parameter yields a notion of (semi\nobreakdash-)stability for elements of $\repspace{\bd}$ as follows.

\begin{definition}
    A representation $M$ is $\theta$-\emph{semistable} (respectively $\theta$-\emph{stable}) if $\theta(M)=0$ and $\theta(M') \leq 0$ (respectively $\theta(M') < 0$) holds for all proper nonzero subrepresentations $M'$ of $M$.
\end{definition}

The sets $\repspace[\stable{\theta}]{\mathbf{d}} \subseteq \repspace[\semistable{\theta}]{\mathbf{d}}$
of $\theta$-stable and $\theta$-semistable representations
are possibly empty Zariski open subsets of $\repspace{\bd}$.

\begin{remark}
A stability parameter $\theta$ with integral coordinates defines a character $\chi_\theta$
of the group $\pgroup{\mathbf{d}}$
by letting $\chi_\theta(g)\colonequals \prod_{i \in Q_0} {\det(g_i)}^{-\theta_i}$.
Such a character provides a linearisation of the trivial line bundle on $\repspace{\mathbf{d}}$,
which we denote by $L(\chi_\theta)$.
King showed \cite[Proposition~3.1]{King:94} that for quiver representations,
(semi\nobreakdash-)stability with respect to this line bundle is equivalent to
the notion of $\theta$-(semi\nobreakdash-)stability defined above.
\end{remark}

The problem studied in this article is to find computational methods
for understanding the dependence of $\repspace[\semistable{\theta}]{\mathbf{d}}$ on $\theta$.
A first question in this direction is to determine for which choices of $\theta$
the set~$\repspace[\semistable{\theta}]{\mathbf{d}}$ is nonempty.
In fact, the set of such $\theta$ forms a cone, which we call the \textit{semistable cone} of $\bd$.

\begin{definition}
The \textit{semistable cone} of a dimension vector $\bd \in \Z^{Q_0}_{\geq 0}$ is
\begin{equation}
\sstCone{\bd} \colonequals \{\theta \in \stabSpace{\bd} \mid \repspace[\semistable{\theta}]{\mathbf{d}} \neq \emptyset\}.
\end{equation}
\end{definition}

Thanks to work of Schofield and Derksen--Weyman,
the semistable cone $\sstCone{\bd}$ can be computed via a recursive algorithm.
Let us detail. Let $\mathbf{e}$ be a dimension vector with $e_i \leq d_i$ for every $i \in Q_0$.
We write $\mathbf{e} \leq \mathbf{d}$ and call $\mathbf{e}$ a \emph{subdimension vector} of $\mathbf{d}$.
We say, following Schofield~\cite[Section~1]{Schofield:92},
that a property holds for a \emph{general representation} of dimension vector $\mathbf{d}$
if there exists a non-empty Zariski open subset of $\repspace{\mathbf{d}}$
such that the property holds for all members of this open subset.

\begin{definition}
    A subdimension vector $\mathbf{e}$ of $\mathbf{d}$ is called a \emph{generic subdimension vector}
    if a general representation of dimension vector $\mathbf{d}$ admits
    a subrepresentation of dimension vector $\mathbf{e}$. In this case we write $\be \hookrightarrow \bd$.
\end{definition}

\begin{example}\label{ex:base}
The zero vector $\mathbf{0}$ is generic for every $\bd$. The subdimension vector $\bd$ is also generic for $\bd$.
\end{example}

The following characterization of generic subdimension vectors
is due to Schofield; see~\cite[Theorems~3.3~and~5.4]{Schofield:92}.
It uses the \textit{Euler form} $\langle  \cdot, \cdot \rangle: \Z^{Q_0} \times \Z^{Q_0} \to \Z$, given by
\begin{equation}
\langle \mathbf{x}, \mathbf{y} \rangle
\colonequals
\sum_{i \in Q_0} \mathbf{x}_i \mathbf{y}_i - \sum_{a \in Q_1} \mathbf{x}_{s(a)}\mathbf{y}_{t(a)}.
\end{equation}

\begin{theorem}\label{t:generic_equivalent}
For a subdimension vector $\mathbf{e}$ of $\mathbf{d}$, the following are equivalent:
\begin{enumerate}
\item [(i)] The vector $\mathbf{e}$ is a generic subdimension vector of $\mathbf{d}$.
\item [(ii)] Every representation in $\repspace{\bd}$ has a subrepresentation of dimension vector $\be$.
\item [(iii)] For all generic subdimension vectors $\mathbf{f}$ of $\mathbf{e}$ we have $\langle \mathbf{f},\mathbf{d}-\mathbf{e} \rangle \geq 0$.
\end{enumerate}
\end{theorem}
The last item gives a recursive algorithm for computing the generic subdimension vectors of $\bd$,
the base case being covered by \cref{ex:base}.
The semistable cone can be characterized in terms of generic subdimension vectors.
The following result appears in~\cite[Theorem~3]{DW:00} for integer-valued
stability parameters and quivers without oriented cycles, so we give a proof here.

\begin{theorem}\label{t:semistable_generic}
Let $\bd \in \Z^{Q_0}_{\geq 0}$ be a dimension vector. Then
\begin{equation}
\sstCone{\bd} = \{\theta \in \stabSpace{\bd} \mid \theta(\be) \leq 0\text{ for all }\be \hookrightarrow \bd \}.
\end{equation}
\end{theorem}

\begin{proof}
Suppose $\repspace[\semistable{\theta}]{\mathbf{d}}$ is not empty.
Then there exists a $\theta$-semistable representation $M$ of dimension vector $\mathbf{d}$.
By \cref{t:generic_equivalent}, it has a subrepresentation of dimension vector $\mathbf{e}$
and by semistability, $\theta(\mathbf{e}) \leq 0$ must hold.

Conversely, suppose that $\theta(\mathbf{e}) \leq 0$ for all $\be \hookrightarrow \bd$.
For any subdimension vector $\mathbf{f}$ of $\bd$ that is not generic,
the set of representations $M \in \repspace{\mathbf{d}}$
which admit a subrepresentation of dimension vector $\mathbf{f}$
can be shown to be a proper closed subset of $\repspace{\bd}$ (see the paragraph before~\cite[Lemma~3.1]{Schofield:92}).
The complement of the finite union of all such subsets
is thus open and dense, and all of its points are $\theta$-semistable by construction.
\end{proof}

It is possible that the set of generic subdimension vectors of $\bd$
is the set $\{\mathbf{0}, \bd\}$; in this case, we have $\sstCone{\bd} = \stabSpace{\bd}$ (see \cref{ex:1}).

\section{Walls}
\label{s:walls}

A representation $M$ of dimension vector $\bd$ is \textit{strictly semistable}
for $\theta \in \stabSpace{\bd}$ if it is $\theta$-semistable but not $\theta$-stable.
Our first objective is to give an algorithm to determine
the set of~$\theta \in \stabSpace{\bd}$ for which $\repspace[\semistable{\theta}]{\mathbf{d}}$ contains a strictly semistable point.
With this in mind, we define sets $W_\be$ whose union will be shown
to be the set of $\theta \in \stabSpace{\bd}$ such that $\repspace[\semistable{\theta}]{\mathbf{d}}$
contains a strictly semistable point.
If $\be$ is a subdimension vector of $\bd$ such that $e_i < d_i$ for some $i \in Q_0$,
we say that $\be$ is a \textit{proper} subdimension vector of $\bd$.

\begin{definition}
Let $\mathbf{e}$ be a nonzero proper subdimension vector of $\mathbf{d}$. Define
\begin{equation}\notag
\wall{\mathbf{e}} \colonequals \{\theta \in \stabSpace{\bd} \mid \theta(\be)=0,\text{ and there exists } (N, M) \in \repspace{\mathbf{e}}\times \repspace[\semistable{\theta}]{\mathbf{d}}\text{ such that } N \subseteq M \}.
\end{equation}
\end{definition}

In other words, $W_\be$ is the set of $\theta \in \sstCone{\bd}$ with a semistable representation
destabilized by a subrepresentation of dimension $\be$. From this the following lemma is immediate.

\begin{lemma}\label{c:artin}
Let $\theta \in \stabSpace{\mathbf{d}}$.
Then $\repspace[\semistable{\theta}]{\mathbf{d}}$ contains a strictly semistable point
if and only if $\theta \in W_\be$ for some nonzero proper subdimension vector $\be$ of $\bd$.
\end{lemma}

Our first result is the following,
which in combination with \cref{t:semistable_generic} allows to effectively compute
the locus in $\sstCone{\mathbf{d}}$ of stability parameters with strictly semistable points.

\begin{proposition}\label{lem:walls}
Let $\be$ be a nonzero proper subdimension vector of $\bd$. Then
\begin{equation}
\wall{\mathbf{e}} = \sstCone{\mathbf{e}} \cap \sstCone{\mathbf{d}-\mathbf{e}}.
\end{equation}
\end{proposition}

The proof uses the following lemma.
\begin{lemma}\label{lem:exact}
Let $0 \to M' \to M \to M'' \to 0$ be a short exact sequence of representations of $Q$.
Then, $M'$ and $M''$ are $\theta$-semistable if and only if $M$ is $\theta$-semistable and $\theta(M')=0$.
\end{lemma}
\begin{proof}
Assume $M$ is $\theta$-semistable and $\theta(M')=0$. Since every subrepresentation of $M'$ is a subrepresentation of $M$, we see that $M'$ is $\theta$-semistable. By additivity of $\theta$ we have $\theta(M'')=0$. Moreover, a subrepresentation $N''$ of $M''$ lifts to a subrepresentation $N$ of $M$ fitting into an exact sequence
\begin{equation}
0 \to M' \to N \to N'' \to 0.
\end{equation}
By semistability of $M$ we have $\theta(N) \leq 0$, and by additivity of $\theta$ we have $\theta(N'')\leq 0.$ So $M''$ is $\theta$-semistable.

Conversely assume $M'$ and $M''$ are $\theta$-semistable. Then $\theta(M) = \theta(M') + \theta(M'')=0$. Now let $N \subset M$ be a subrepresentation, let $N'$ be the preimage in $M'$, and let $N''$ be the image in $M''$. We have a short exact sequence
\begin{equation}
0 \to N' \to N \to N'' \to 0,
\end{equation}
and since $M'$ and $M''$ are semistable we have $\theta(N') \leq 0$ and $\theta(N'')\leq 0$. It follows that $\theta(N)\leq 0$ and $M$ is $\theta$-semistable.
\end{proof}

\begin{proof}[Proof of \cref{lem:walls}]
Suppose $\theta \in \sstCone{\be} \cap \sstCone{\bd - \be}$.
Then $ \theta(\be)=0$, and we have representations $M'$ and $M''$
of dimension vectors $\be$ and $\bd - \be$, respectively, that are $\theta$-semistable.
It follows from \cref{lem:exact} that $M' \oplus M''$ is $\theta$-semistable of dimension vector $\bd$ and has a subrepresentation $M'$ of dimension vector $\be$.

Conversely, suppose that $\theta$ is in $W_\be$, i.e.,
there is a $\theta$-semistable representation $M$
of dimension vector $\bd$ having a subrepresentation $M'$
of dimension vector $\be$, and $\theta(\be)=0$.
It follows from \cref{lem:exact} that $M'$ is $\theta$-semistable
of dimension vector $\be$ and $M/M'$ is $\theta$-semistable
of dimension vector $\bd - \be$.
So $\theta \in \sstCone{\be} \cap \sstCone{\bd - \be}$ as claimed.
\end{proof}

We note that $W_\be$ is always contained in the subspace $H_\be$ of $\stabSpace{\bd}$, where
\begin{equation}
H_\be \colonequals \{\theta \in \stabSpace{\bd} \mid \theta(\be) = 0\} = \stabSpace{\bd} \cap \stabSpace{\be}.
\end{equation}
\cref{ex:1,ex:cycle} are cases where $W_\be$
is a proper subset of $H_\be\cap \sstCone{\bd}$.
In general, $W_{\be}$ can only be a proper subset
of $H_\be \cap \sstCone{\bd}$ if $\be$ is not
a generic subdimension vector,
as shown below.

\begin{lemma}\label{L:3.3}
Let $\be$ be a generic proper nonzero subdimension vector of $\bd$. Then
\begin{equation}
W_\be = H_\be \cap \sstCone{\bd}.
\end{equation}
\end{lemma}
\begin{proof}
The inclusion $W_\be \subseteq H_\be$ is immediate.
For the reverse inclusion, consider a stability parameter~$\theta\in{H_\be}\cap\sstCone{\bd}$.
Then $\theta(\be) = 0$ and $\theta \in \stabSpace{\bd}$,
and $\repspace[\semistable{\theta}]{\mathbf{d}}$ is not empty.
Since $\be$ is generic, every representation in $\repspace[\semistable{\theta}]{\mathbf{d}}$
has a subrepresentation of dimension vector $\be$.
In particular, we have a pair $(N, M) \in \repspace{\mathbf{e}} \times \repspace[\semistable{\theta}]{\mathbf{d}}$,
showing that $\theta \in W_\be$.
\end{proof}

\section{GIT equivalence}

Two stability parameters $\theta, \eta$ are \emph{GIT equivalent} if $\repspace[\semistable{\theta}]{\mathbf{d}}=\repspace[\semistable{\eta}]{\mathbf{d}}$.
In this section we give an algorithm to determine
when $\theta$ and $\eta$ are GIT equivalent.
Our method uses the walls $W_\be$ from \cref{s:walls} and general properties of twisted affine GIT.
\begin{definition}
Let $\mathbf{d}$ be a dimension vector,
let $L \subset \stabSpace{\bd}$ be
an affine subspace
and let $\theta \in L$.
We say $\theta$ is \emph{L-generic} if there exists $\epsilon > 0$
such that $\eta$ and $\theta$ are GIT equivalent for all $\eta \in L$ which satisfy $\left\| \theta - \eta \right\| < \epsilon$.
Otherwise, we say $\theta$ is \emph{L-special}.
\end{definition}

\begin{proposition}\label{t:main}
Let $L \subset \stabSpace{\bd}$ be
an affine subspace
and let $\theta \in L$.
Then $\theta$ is~$L\text{\hyphen}$special if and only if there exists
a nonzero proper subdimension vector $\be$ of $\bd$
such that $\theta \in W_\be$ and $L$ meets $H_\be$ transversely.
\end{proposition}

The proof uses the following lemma, which is well-known in the setting of twisted affine GIT.
To be self-contained we give a proof in the language of quiver representations.

\begin{lemma}\label{l:sst_locus_gets_smaller_in_neighborhood}
Let $\theta \in \stabSpace{\mathbf{d}}$. Then there exists $\epsilon > 0$
such that for all $\eta \in \stabSpace{\mathbf{d}}$ with $\left\| \theta-\eta \right\| < \epsilon$,
we have $\repspace[\semistable{\theta}]{\mathbf{d}} \supseteq \repspace[\semistable{\eta}]{\mathbf{d}}$.
\end{lemma}

\begin{proof}
    The complement of the semistable cone is open,
    so we may assume without loss of generality that $\theta \in \sstCone{\mathbf{d}}$.
    Define the set
    $I \colonequals \{\mathbf{e} \in \smash{\mathbb{N}^{Q_0}\setminus\{\mathbf{0}\}} \mid \mathbf{e} \leq \mathbf{d},~\mathbf{e} \notin \Q\cdot\mathbf{d}\}$
    and let $J \subseteq I$ be the subset of all $\mathbf{e}$ such that $\theta$ lies on the hyperplane $H_\mathbf{e}$. Then $\theta$ has a positive distance to all hyperplanes $H_\mathbf{f}$ with $\mathbf{f} \in I\setminus J$. Therefore we find $\epsilon > 0$ such that $\bigcup_{\mathbf{f} \in I\setminus J}H_\mathbf{f}$ does not intersect the $\epsilon$-ball around $\theta$. Let $\eta$ be contained in this $\epsilon$-ball and let $M$ be an $\eta$-semistable representation of dimension vector $\mathbf{d}$. Let $U$ be a subrepresentation of $M$ and let $\mathbf{e}$ be its dimension vector. Then $\eta(\mathbf{e}) \leq 0$.

    If $\eta(\mathbf{e}) = 0$ then $\mathbf{e}$ is either a fraction of $\bd$, or it belongs to $J$,
    because $\eta$ is not contained in $\bigcup_{\mathbf{f} \in I\setminus J}H_\mathbf{f}$.
    In both cases, we have $\theta(\mathbf{e}) = 0$.
    If $\eta(\mathbf{e}) < 0$, then we claim $\theta(\mathbf{e}) \leq 0$:
    indeed, if $\theta(\be) > 0$ then $\mathbf{e} \not \in J$ but $\theta$ and $\eta$
    are on opposite sides of the hyperplane $H_\mathbf{e}$; this implies that $H_\mathbf{e}$ cuts through
    the $\epsilon$-ball around $\theta$, which is a contradiction.
    In either case, we have shown that $\theta(U) \leq 0$ which implies that $M$ is $\theta$-semistable.
\end{proof}

\begin{proof}[Proof of \cref{t:main}]
Suppose that $\theta$ is $L$-special. If $\eta \in L$ is a stability parameter not GIT-equivalent to $\theta$, then there is a subdimension vector $\be$ such that $\theta(\be)\leq 0$ and $\eta(\be) > 0$, or vice versa. Note that for such $\be$ the hyperplane $H_\be$ meets $L$ transversely. It follows that $\theta$ must be contained in $H_\be$ for some $H_\be$ meeting $L$ transversely, as if it was not, the complement of such $H_\be$ would contain a neighborhood of $\theta$ consisting of stability parameters GIT-equivalent to $\theta$,
and $\theta$ would not be $L$-special.

Let $J$ be the set of subdimension vectors $\be$ such that $\theta \in H_\be$. Choose $\epsilon$ so that if $\be \not \in J$ then $H_\be$ does not meet the $\epsilon$-ball around $\theta$, and let $\eta \in \stabSpace{\bd}$ such that $||\theta - \eta|| < \epsilon$ but $\eta$ and $\theta$ are not GIT equivalent. It follows from \cref{l:sst_locus_gets_smaller_in_neighborhood} that there is a representation $M$ of dimension vector $\bd$ that is $\theta$-semistable but not $\eta$-semistable. Thus there is a subrepresentation $M'$ of $M$ of dimension vector $\be$ such that $\theta(M') \leq 0$ but $\eta(M')>0$. If $\theta(M') < 0,$ then $\theta$ is not in $H_\be$, so this hyperplane does not meet the $\epsilon$-ball containing $\theta$ and $\eta$, contradicting the fact that $\eta(M')>0$. So we must have $\theta(M')=0$. It follows from \cref{lem:walls} that $\theta$ is in $W_\be$.

Conversely, let $\be$ be a subdimension vector with $\theta \in W_\be$ and $H_\be$ meeting $L$ transversely. We show that $\theta$ is $L$-special. Let $\epsilon > 0$. Choose $\eta $ in $\stabSpace{\bd} \setminus H_\be$ such that $||\theta - \eta || < \epsilon$; this is possible as $H_\be$ has positive codimension inside $\stabSpace{\bd}$. Then $\eta(\be) \neq 0$. As $\theta \in \wall{\mathbf{e}}$, there exist $M' \in \repspace[\semistable{\theta}]{\mathbf{e}}$ and $M'' \in \repspace[\semistable{\theta}]{\mathbf{d}-\mathbf{e}}$. Then $M = M' \oplus M''$ is $\theta$-semistable. However it is not $\eta$-semistable: since $\eta(\bd)=0$ and $\eta(\be) \neq 0$, either $\eta(\be)>0$ or $\eta(\bd - \be)>0$. In the first case $M'$ is a subrepresentation with $\eta(M') > 0$, and in the second $M''$ is a subrepresentation with $\eta(M'')>0$. This shows that $\theta$ is special.
\end{proof}

By~\cite[Theorem~3.2]{AH:09}
the GIT equivalence classes on $\sstCone{\bd}$ define a fan, called \textit{the GIT fan of} $(Q, \bd)$.
 In particular, each GIT equivalence class is convex. Given $\theta, \eta \in \stabSpace{\bd}$, denote their convex hull by
\begin{equation}
[\theta, \eta] \colonequals \{t\theta + (1-t)\eta \mid t \in [0, 1]\}.
\end{equation}
\begin{corollary}\label{C:GIT}
Two stability parameters $\theta, \eta \in \sstCone{\bd}$ are GIT equivalent
if and only if for every nonzero proper $\be \leq \bd$,
either $[\theta, \eta]$ is contained in $W_{\be}$ or $[\theta, \eta]$ does not meet $W_{\be}$.
\end{corollary}
\begin{proof}
The intersection of the GIT fan with the line segment $[\theta, \eta]$ is a decomposition of the segment into a collection of points and relatively open intervals that are precisely the GIT equivalence classes for points on $[\theta, \eta]$.
The points in this decomposition are precisely the $L$-special stability parameters in $[\theta, \eta]$, where $L$ is the line containing $[\theta, \eta]$, and hence by \cref{t:main} are precisely the transverse intersections of $W_{\be}$ with $[\theta, \eta]$ for subdimension vectors $\be$ of $\bd$.

We see that $\theta$ and $\eta$ are GIT equivalent if and only if the decomposition of $[\theta, \eta]$ given
above
contains no points, and this holds if and only if there are no transverse intersections of $W_{\be}$ with $[\theta, \eta]$. But this is equivalent to $[\theta, \eta]$ either being contained in $W_\be$ or not meeting $W_\be$.
\end{proof}

Following \cite[p.163]{AH:09} we say that a polyhedral fan $\Sigma$ is \textit{determined by its walls} if, letting $\omega$ be the support of $\Sigma$ and $\eta_1, \ldots, \eta_s$ the facets of the full-dimensional cones of $\Sigma$, we have that the maximal cones of $\Sigma$ are the closures of the connected components of $\omega \setminus (\eta_1 \cup \ldots \cup \eta_s)$. In this case, the maximal cones all have dimension $\dim(\omega)$, and the cones $\eta_1, \ldots, \eta_s$ determine $\Sigma$. It follows that $\Sigma$ is determined by its walls if and only if every cone of $\Sigma$ is a face of some full-dimensional cone of $\Sigma$. By \cite[Proposition~5.2]{AH:09}, the GIT fan of $(Q, \bd)$ is determined by its walls.

\begin{corollary}\label{c:walls}
Let $\Sigma$ be the GIT fan of $(Q, \bd)$. The union of the facets of the maximal cones of $\Sigma$ is equal to the union of $W_\be$ for $\be \leq \bd$ satisfying $\dim W_\be = \dim \sstCone{\bd}-1$.
\end{corollary}
\begin{proof}
Since $\Sigma$ is determined by its walls, from \cref{t:main} we have that the union of the facets of the maximal cones of $\Sigma$ is equal to the union of $W_\be$ for proper nonzero $\be \leq \bd$. This is equal to the union of $W_\be$ for $\be \leq \bd$ satisfying $\dim W_\be = \dim \sstCone{\bd}-1$ since $\Sigma$ is determined by its walls.
\end{proof}

Finally, observe that a polyhedral cone $\omega$ and a collection of hyperplanes determine
a unique fan that is determined by its walls.
For $\be \leq \bd$, let $W_{\overline{\be}}$ be the union over $k \in \mathbb{Z}_{>0}$
of the walls $W_{k\be}$ for which $k\be \leq \bd$
and $\dim W_{k\be} = \dim \sstCone{\bd}-1$.
\footnote{A priori, $W_{\overline{\be}}$ may not be convex.}
We have proved the following effective algorithm for computing the GIT fan of $(Q, \bd)$.
\begin{corollary}\label{C:fan}
The following algorithm computes the GIT fan of $(Q, \bd)$:
\begin{enumerate}
\item Let $\Sigma'$ be the fan determined by its walls arising from the cone $\sstCone{\bd}$ and hyperplanes $H_\be$ for those indivisible $\be$ such that $W_{\overline{\be}}$ is nonempty.
\item Let $\Sigma$ be the fan obtained from $\Sigma'$ by merging all the full-dimensional cones that share a facet contained in $H_\be$ but not in $W_{\overline{\be}}$.
\end{enumerate}
Then the GIT fan of $(Q, \bd)$ is equal to $\Sigma$.
\end{corollary}
\begin{proof}
The GIT fan and $\Sigma$ are both determined by their walls, and by \cref{c:walls} the union of the facets of the maximal cones for these two fans is identical.
\end{proof}

\section{Criterion for existence of a geometric phase}

In this section we use our description of the walls $W_\be$ to characterize pairs $(Q, \bd)$ having a geometric phase.
Recall that a dimension vector $\bd$ is a \textit{Schur root}
if a general representation of dimension vector $\bd$ is indecomposable,
and it is \textit{indivisible} if~$\gcd(\bd)=1$.
\begin{theorem}\label{T:phase}
Let $Q$ be a quiver and $\bd$ a dimension vector.
The following are equivalent.
\begin{itemize}
\item[(i)] The pair $(Q, \bd)$ has a geometric phase.
\item[(ii)] The dimension vector $\bd$ is a Schur root and indivisible.
\item[(iii)] The cone $\sstCone{\bd}$ spans $\stabSpace{\bd}$, and every wall $W_\be$ spans a proper subspace of $\stabSpace{\bd}$.
\end{itemize}
\end{theorem}

We note that large portions of \cref{T:phase} follow from King's results~\cite[Propositions~4.4 and~5.3]{King:94}, and both our formulation of this theorem and our proof rely heavily on King's ideas. We also point out that the equivalence of (i) and (iii) is a special property of VGIT problems coming from representations of quivers: the analogous statements for more general linear representations of complex reductive groups are not equivalent, as is evident from \cref{ex:nonquiver}.
Our proof uses the notion of \textit{generic ext}. 
Using a theorem of Schofield (see \cite[Theorem~5.4]{Schofield:92}),
we can define the quantity $\operatorname{ext}(\bd, \be)$
for the dimension vectors $\bd$ and $\be$ as
\begin{equation}\label{eq:ext}
\operatorname{ext}(\bd, \be) \colonequals \max_{\bd' \hookrightarrow \bd}\{-\langle\bd', \be \rangle \} = \max_{\be \hookrightarrow \be'}\{-\langle \bd, \be - \be' \rangle \}.
\end{equation}
\begin{remark}\label{R:ext-linear}
For all dimension vectors $\bd, \be$ and positive integers $k$ we have
\begin{equation}
k\operatorname{ext}(\bd, \be) = \operatorname{ext}(\bd, k\be) = \operatorname{ext}(k\bd, \be).
\end{equation}
This follows from the linearity of the Euler form $\langle \cdot, \cdot \rangle$ and the respective equalities in~\eqref{eq:ext}.
\end{remark}

The next lemma follows from results in~\cite{DW:00},
but since op. cit. works with integer-valued stability parameters
and quivers without oriented cycles we give a self-contained proof here.

\begin{lemma}\label{L:saturation}
For any positive integer $k$, we have the identity $\sstCone{\bd} = \sstCone{k\bd}$.
\end{lemma}

\begin{proof}
Let \(\theta \in \sstCone{\bd}\).
This means that there exists a \(\theta\)-semistable representation \(M\) of dimension vector \(\bd\).
Then \(M^{\oplus k}\), the direct sum of \(k\) copies of \(M\), is \(\theta\)-semistable by \cref{lem:exact}.
This shows that \(\theta \in \sstCone{k\bd}\).

For the converse inclusion, let \(\theta \in \sstCone{k\bd}\).
By \cref{t:semistable_generic}, it is enough to show that \(\theta(\be) \leq 0\) for all generic subdimension vectors \(\be\) of \(\bd\).
Let then \(\be\) be such a generic subdimension vector of \(\bd\).
By~\cite[Theorem~3.3]{Schofield:92},
$\be \hookrightarrow \bd$ is equivalent to \(\operatorname{ext}(\be,\bd-\be) = 0\),
which implies \(\operatorname{ext}(k\be,k\bd-k\be) = 0\)
by \cref{R:ext-linear}.
Hence \(k\be\) is a generic subdimension vector of \(k\bd\). As \(\theta \in \sstCone{k\bd}\), we have
\begin{equation}
0 \geq \theta(k\be) = k\theta(\be),
\end{equation}
and as \(k \neq 0\), we obtain \(\theta(\be) \leq 0\) as desired.
\end{proof}

\begin{proof}[Proof of \cref{T:phase}]
To show that (i) implies (ii), we will show that if either $\bd$ is divisible or not a Schur root, then there is some proper nonzero subdimension vector $\be$ of $\bd$ such that $W_\be = \sstCone{\bd}$. First assume $\bd$ is divisible, so $\bd = k\be$ for some integer $k \geq 2$. Then

\begin{equation}
W_\be = \sstCone{\be} \cap \sstCone{\bd - \be} =  \sstCone{\be} \cap \sstCone{(k-1)\be} = \sstCone{\bd}
\end{equation}
where the last equality holds by \cref{L:saturation}.
Now assume that $\bd$ is not a Schur root, so the generic representation of dimension vector $\bd$ decomposes as a direct sum of representations.
The dimension vectors of this decomposition give a pair $(\be, \bd-\be)$ of proper nonzero generic subdimension vectors of $\bd$, and hence
\begin{equation}
\sstCone{\bd} \subseteq \{\theta \in \stabSpace{\bd} \mid \theta(\be) \leq 0\text{ and } \theta(\bd - \be) \leq 0\} = H_\be.
\end{equation}
But $W_\be = H_\be \cap \sstCone{\bd}$ by \cref{L:3.3}, so $W_\be = \sstCone{\bd}$.

To show that (ii) implies (iii), assume that $\bd$ is indivisible and a Schur root.
Since $\bd$ is a Schur root, by \cite[Theorem~6.1]{Schofield:92}, for all proper nonzero $\be \hookrightarrow \bd$ we have
\begin{equation}\notag
\langle \bd, \be \rangle - \langle \be, \bd \rangle < 0.
\end{equation}
Hence for the so-called canonical stability parameter $\theta_{\bd} (\bf) \colonequals \langle \bd, \bf\rangle - \langle \bf, \bd \rangle$, a general representation of dimension vector $\bd$ is $\theta_{\bd}$-stable.\footnote{Note that $\theta_{\bd}(\bd) = 0$, so $\theta_\bd \in \stabSpace{\bd}$.}
Now let $B_\epsilon$ denote the open ball in $\stabSpace{\bd}$ centered at $\theta_\bd$ of radius $\epsilon$. We can choose $\epsilon$ small enough so that $\eta(\be)<0$ for all $\eta \in B_\epsilon$ and proper nonzero $\be \hookrightarrow \bd$.
In particular, a general representation of dimension vector $\bd$ is $\eta$-stable and hence $B_\epsilon \subset \sstCone{\bd}.$ It follows that $\sstCone{\bd}$ spans $\stabSpace{\bd}.$
Finally, to see that every wall $W_\be$ spans a proper subspace of $\stabSpace{\bd}$, suppose to the contrary that for some nonzero proper $\be \leq \bd$, the wall $W_\be$ spans $\stabSpace{\bd}$.
Then $H_\be = H_\bd$ and $k\be = \ell\bd$ for some coprime integers $k$, $\ell$.
Since $\bd$ is indivisible, we obtain $k = 1$, and since $\be \leq \bd$ we obtain $\ell=1$. So $\be$ is not a proper subdimension vector of $\bd$.

Finally, that (iii) implies (i) is immediate.
\end{proof}

\section{The wall system and special subdimension vectors}
\label{subsection:special-subdimension-vectors}
Recall that a representation of $Q$ of dimension vector $\bd$ is
\textit{stable} if it is $\theta$-stable for some $\theta \in \stabSpace{\bd}$.
\begin{definition}
A subdimension vector $\be \leq \bd$ is \textit{special} for $Q$ if
it is not generic and some stable representation of $Q$ of dimension vector $\bd$ has a subrepresentation of dimension vector $\be$.
\end{definition}
\begin{remark}
This agrees with the definition given in \cite[p.213]{HP:02} for acyclic quivers.
\end{remark}
\noindent
Note that if $\be \leq \bd$ is special, then $\bd - \be$ may or may not be special. \cref{ex:walls-strictly-contained-hyperplanes} shows that both phenomena occur.

The main result of this section is \cref{L:special}, which gives a computable characterization of special subdimension vectors of $(Q, \bd)$.

\begin{proposition}\label{L:special}
Let $(Q, \bd)$ be a quiver and dimension vector. If $\bd$ is not a Schur root of $Q$ then no special subdimension vectors exist. If $\bd$ is a Schur root, then $\be \leq \bd$ is special if and only if $\be$ is not generic and the cone
\begin{equation}
C_\be \colonequals \{ \theta \in \stabSpace{\bd} \mid \theta(\be') \leq 0 \text{ and } \theta(\bf' + \be) \leq 0\text{ for all }\be' \hookrightarrow \be \text{ and } \mathbf{f'}\overset{\neq}{\hookrightarrow}\bd-\be \}
\end{equation}
is not contained in any of its defining hyperplanes $H_{\be'}$ or $H_{\bf' + \be}$.
Here we write $\mathbf{f'}\overset{\neq}{\hookrightarrow}\bd-\be$ for a proper generic subdimension vector $\bf'$ of $\bd - \be$.
\end{proposition}
\begin{proof}
First assume that $\bd$ is not a Schur root.
Then no stable representations exist, so no $\be \leq \bd$ can be special.
Indeed, if $\bd$ is not a Schur root then we have a pair
of proper nonzero generic subdimension vectors $(\be, \bd - \be)$ of $\bd$,
and hence if $M$ is a $\theta$-stable representation for some $\theta \in \stabSpace{\bd}$,
we must have both $\theta(\be) < 0 $ and $\theta(\bd - \be)< 0$, which is impossible.
For the remainder of the proof we assume $\bd$ is a Schur root.

Assume that $\be$ is not generic and that $C_\be$ is not contained
in any of its defining hyperplanes $H_{\be'}$ or $H_{\bf' + \be}$.
Then we can find $\theta \in C_\be$ such that $\theta(\be) < 0$, and moreover
\begin{equation}\label{eq:inequalities}
\theta(\be') < \theta(\be) \quad\quad\text{   and   } \quad\quad\theta(\bf') < \theta(\bd - \be)
\end{equation}
for all $\mathbf{e'}\overset{\neq}{\hookrightarrow}\be$ and $\mathbf{f'}\overset{\neq}{\hookrightarrow}\bd-\be$.
To show that $\be$ is special, it is enough to construct a~$\theta\text{\hyphen}$stable representation
of dimension vector $\bd$ with a subrepresentation of dimension vector $\be$.
Let $M' \in \repspace{\be}$ and $M'' \in \repspace{\mathbf{d-e}}$
be representations that only admit subrepresentations of
generic subdimension vectors.
Since $\be$ is not a generic subdimension vector of $\bd$,
we can pick a nontrivial extension $M \in \repspace{\bd}$ of $M''$ by $M'$.
Then $M$ has dimension vector $\bd$ and it admits a subrepresentation $M'$ of dimension vector $\be$.
We claim that $M$ is $\theta$-stable, and hence $\be$ is a special subdimension vector.
To see this, let $N \subset M$ be a proper nonzero subrepresentation.
Since $M$ is a nontrivial extension, $N$ is not equal to $M''$,
and our choices of $M'$ and $M''$ ensure that
\begin{equation}
\dim (N \cap M') = \be'\quad\quad \text{   and   }\quad\quad \dim(N/(N \cap M')) = \bf'
\end{equation}
for some $\mathbf{e'}{\hookrightarrow}\be$ and $\mathbf{f'}\overset{\neq}{\hookrightarrow}\bd-\be$.
Then $\theta(N)<0$ follows from additivity of $\theta$ and \eqref{eq:inequalities}.

Conversely, if $\be$ is generic it is not special.
Assume then that $\be$ is not generic and that $C_\be$ is contained
in a defining hyperplane $H_{\be'}$ or $H_{\bf' + \be}$,
for some $\be'\hookrightarrow\be$ or
for some $\mathbf{f}'\overset{\neq}{\hookrightarrow}\bd-\be$.
This is equivalent to saying that for every $\theta \in \stabSpace{\bd}$
one of the following holds:
\begin{itemize}
\item[(i)]there is some $\be'\hookrightarrow \be$ such that $\theta(\be') \geq 0$, or
\item[(ii)] there is some $\mathbf{f'}\overset{\neq}{\hookrightarrow}\bd-\be$ such that $\theta(\bf'+\be) \geq 0$.
\end{itemize}
Now let $M$ be a representation of dimension vector $\bd$ and $M' \subset M$ be a subrepresentation of dimension vector $\be$.
To show that $\be$ is not special we must show that $M$ is not~$\theta\text{\hyphen}$stable for any $\theta \in \stabSpace{\bd}$.
For any $\be'\hookrightarrow\be$, the representation $M'$ has a subrepresentation of dimension vector $\be'$
and hence $M$ does as well, so if (i) holds then $M$ is not stable.
Similarly, for any $\bf'\hookrightarrow\bd - \be$, the representation $M/M'$
has a subrepresentation of dimension vector $\bf'$, and its preimage in $M$
is a subrepresentation of dimension vector $\bf' + \bd$,
so if (ii) holds then $M$ is not stable either.
\end{proof}
It is shown in \cite[Lemma~3.3]{HP:02} that each hyperplane in the wall system
of $(Q, \bd)$ is orthogonal to a special subdimension vector.
We can use \cref{L:special} to show that the converse does not hold;
that is, if $\be$ is special, $H_{\be}$ need not be a hyperplane in the wall system.
This occurs in \cref{ex:1,ex:walls-strictly-contained-hyperplanes}.

\section{Examples: a gallery of walls}\label{sec:examples}

In this section, we present several examples illustrating
various phenomena that can and can not occur
in a GIT fan of a quiver and dimension vector.

The methods described in this paper
for computing the walls $W_\be$ of $(Q, \bd)$
as well as its GIT fan have been implemented
in \software{SageMath} and \software{Julia}
using the \software{QuiverTools} package \cite{quivertools}.
The code is available at \cite{walls-and-chambers-implementation},
and can reproduce all the results in this section.
To our knowledge,
this is the most convenient method
for computing \cref{ex:non-flag,ex:mutation,ex:segre-cubic}.

The GIT fan in \cref{ex:walls-strictly-contained-hyperplanes}
can also be computed with
the \software{Macaulay2} package \software{ThinSincereQuivers} \cite{ThinSincereQuivers}.

\begin{example}\label{ex:1}
It is possible that $W_\be = H_\be$ and it is also possible that $W_\be \subset H_\be$ is the origin. Also, not every special subdimension vector determines a hyperplane in the wall system.
Let $(Q, \bd)$ be the quiver and dimension vector pictured below left.

\noindent
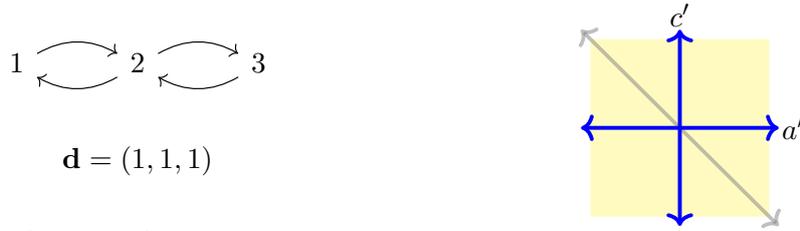
\begin{minipage}{.5\textwidth}
\centering
\begin{tikzcd}[column sep = 5pt]
1 \arrow[r, bend left] & 2 \arrow[l, bend left] \arrow[r, bend left] & 3 \arrow[l, bend left]\\
& \bd = (1,1,1)
\end{tikzcd}
\end{minipage}%
\begin{minipage}{0.5\textwidth}
\centering
\begin{tikzpicture}[scale=.65]
\draw[yellow, fill=yellow, opacity=.25] (-1.8,-1.8) rectangle (1.8,1.8);
\draw[<->, gray, opacity=.5, line width=0.5mm] (-2,0)--(2,0);
\draw[<->, gray, opacity=.5, line width=0.5mm] (0,-2)--(0,2);
\draw[<->, gray, opacity=.5, line width=0.5mm] (-2, 2)--(2, -2);
\draw[<->, blue, line width=0.5mm] (0,-2)--(0,2);
\draw[<->, blue, line width=0.5mm] (-2,0)--(2,0);
\node at (2.3, 0) {$a'$};
\node at (0, 2.3) {$c'$};
\end{tikzpicture}
\end{minipage}

\noindent
Since the definition of $W_\be$ is symmetric in $\be$ and $\bd-\be$, we see that there are three walls: $W_{(1,1,0)}$ and $W_{(0,1,1)}$ both have dimension $1$, while $W_{(1,0,1)}$ has dimension $0$.
Write $(a, b, c)$ for a point in $\R^{Q_0}$. Then $\stabSpace{\mathbf{\bd}}$ is the locus where $a+b+c=0$. We draw this 2-dimensional real vector space at right in the basis $a' = (1,-1,0)$ and $c' = (0, -1, 1)$. The yellow region is $\sstCone{\mathbf{\bd}}$ (it is all of $\stabSpace{\mathbf{\bd}}$), the grey lines are the hyperplanes $H_\be$, and the blue lines are the walls $W_\be$. Finally, every subdimension vector is special for $(Q, \bd)$, but $H_{(1,0,1)}$ does not belong to the wall system.
\end{example}

\begin{example}
\label{ex:walls-strictly-contained-hyperplanes}\label{ex:cycle}
It is possible that $W_\be$ is a proper subset of $H_\be$ but that $W_\be$ still spans $H_\be$, even when the quiver is acyclic. Also, not every special subdimension vector determines a hyperplane in the wall system, and if $\be$ is a special subdimension vector then $\bd - \be$ may or may not be. Let $(Q, \bd)$ be the quiver and dimension vector pictured below left.

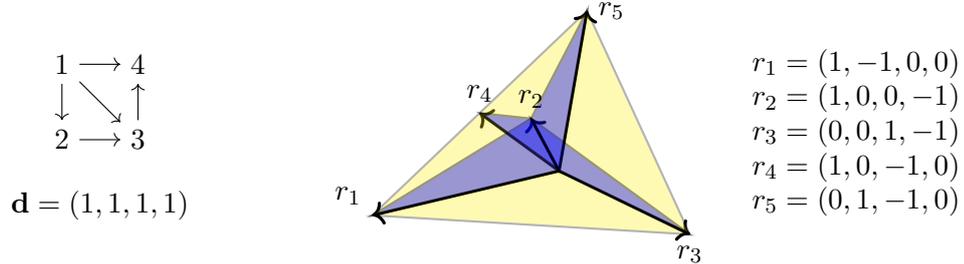
\begin{minipage}{.35\textwidth}
   \centering
      \begin{tikzpicture}

\node (1)                   {1};
\node (2) [below of = 1]    {2};
\node (3) [right of = 2]    {3};
\node (4) [above of = 3]    {4};

\draw[->] (1) edge (2);
\draw[->] (2) edge (3);
\draw[->] (3) edge (4);
\draw[->] (1) edge (3);
\draw[->] (1) edge (4);
\end{tikzpicture}\\
\vspace{.1in}

$\bd = (1,1,1,1)$
    \end{minipage}%
    \begin{minipage}{0.4\textwidth}
    \centering
        \tdplotsetmaincoords{70}{145}
\begin{tikzpicture}
[scale=3, tdplot_main_coords,
axis/.style={->, black},
ray/.style={->, black, very thick},
wall/.style={opacity=0.3, thick, fill=yellow},
inner wall/.style={opacity=0.4, thick, fill=blue}
]

\coordinate (O) at (0,0,0);

\coordinate (r1) at (1, 0, 0){};
\coordinate (r2) at (0.5, 0.5, 0.5){};
\coordinate (r3) at (0, 1, 0){};
\coordinate (r4) at (0.25, -0.25, 0.25){};
\coordinate (r5) at (-0.5, -0.5, 0.5){};


\draw[ray] (O) -- (r1) node[above left] {$r_1$}; 
\draw[ray] (O) -- (r3) node[below] {$r_3$}; 
\draw[ray] (O) -- (r4) node[above] {$r_4$}; 
\draw[ray] (O) -- (r5) node[right] {$r_5$}; 

\draw[wall] (O) -- (r1) -- (r3) -- cycle;
\draw[wall] (O) -- (r3) -- (r5) -- cycle;
\draw[wall] (O) -- (r1) -- (r5) -- cycle;

\draw[ray] (O) -- (r2) node[above] {$r_2$}; 

\draw[inner wall] (O) -- (r1) -- (r2) -- cycle;
\draw[inner wall] (O) -- (r2) -- (r3) -- cycle;
\draw[inner wall] (O) -- (r2) -- (r5) -- cycle;
\draw[inner wall] (O) -- (r2) -- (r4) -- cycle;

\end{tikzpicture}
\end{minipage}
\begin{minipage}{0.2\textwidth}
$r_1 = (1, -1, 0, 0)$

$r_2 = (1, 0, 0, -1)$

$r_3 = (0, 0, 1, -1)$

$r_4 = (1, 0, -1, 0)$

$r_5 = (0, 1, -1, 0)$
\end{minipage}

\noindent
The GIT fan for $(Q, \bd)$ is pictured on the right: the yellow region is $\sstCone{\bd}$
(it is a strictly convex polyhedral cone of dimension $3$) and the blue polygons
are the walls $W_\be$.
For $\be$ equal to $(0,1,1,0)$ or $(0,0,1,0)$ we see that $W_\be \subsetneq H_\be$.
Moreover, the special subdimension vectors for $(Q, \bd)$ are
\begin{equation*}
\{\;
(0, 1, 0, 0),\;
(0, 0, 1, 0),\;
(0, 1, 1, 0),\;
(1, 0, 0, 1),\;
(0, 1, 0, 1),\;
(1, 1, 0, 1),\;
(1, 0, 1, 1)\;
\}.
\end{equation*}
Hence $\be = (0,1,0,1)$ is a special subdimension vector such that $\bd - \be$ is not special. Also, $H_\be$ is not in the wall system.
\end{example}

\begin{example}\label{ex:3}
It is possible that $W_\be$ is all of $\sstCone{\bd}$: let $Q$ be the flag quiver
\[
\begin{tikzcd}
1 \arrow[r] \arrow[r, shift left=2] \arrow[r, shift right=2] & 2 \arrow[r] & 3
\end{tikzcd}.
\]
Several different dimension vectors and the corresponding walls are pictured below.

\noindent
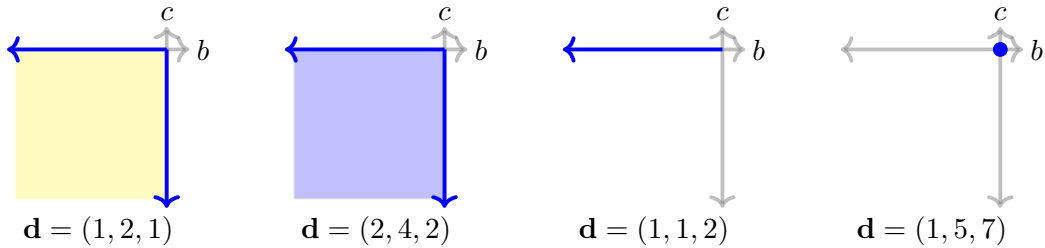
\begin{center}
\begin{minipage}{0.24\textwidth}
\centering
\begin{tikzpicture}[scale=.6]
\draw[yellow, fill=yellow, opacity=.25] (-3.3,-3.3) rectangle (0, 0);
\draw[<->, gray, opacity=.5, line width=0.5mm] (-3.5,0)--(0.5,0);
\draw[<->, gray, opacity=.5, line width=0.5mm] (0,-3.5)--(0,0.5);
\draw[->, blue, line width=0.5mm] (0,0)--(0,-3.5);
\draw[<-, blue, line width=0.5mm] (-3.5,0)--(0,0);
\node at (.8, 0) {$b$};
\node at (0, .8) {$c$};
\node at (-1.5, -4) {$\bd = (1,2,1)$};
\end{tikzpicture}
\end{minipage}
\begin{minipage}{0.24\textwidth}
\centering
\begin{tikzpicture}[scale=.6]
\draw[gray, fill=blue, opacity=.25] (-3.3,-3.3) rectangle (0, 0);
\draw[<->, gray, opacity=.5, line width=0.5mm] (-3.5,0)--(0.5,0);
\draw[<->, gray, opacity=.5, line width=0.5mm] (0,-3.5)--(0,0.5);
\draw[->, blue, line width=0.5mm] (0,0)--(0,-3.5);
\draw[<-, blue, line width=0.5mm] (-3.5,0)--(0,0);
\node at (.8, 0) {$b$};
\node at (0, .8) {$c$};
\node at (-1.5, -4) {$\bd = (2,4,2)$};
\end{tikzpicture}
\end{minipage}
\begin{minipage}{0.24\textwidth}
\centering
\begin{tikzpicture}[scale=.6]
\draw[<->, gray, opacity=.5, line width=0.5mm] (-3.5,0)--(0.5,0);
\draw[<->, gray, opacity=.5, line width=0.5mm] (0,-3.5)--(0,0.5);
\draw[<-, blue, line width=0.5mm] (-3.5,0)--(0,0);
\node at (.8, 0) {$b$};
\node at (0, .8) {$c$};
\node at (-1.5, -4) {$\bd = (1,1,2)$};
\end{tikzpicture}
\end{minipage}
\begin{minipage}{0.24\textwidth}
\centering
\begin{tikzpicture}[scale=.6]
\draw[<->, gray, opacity=.5, line width=0.5mm] (-3.5,0)--(0.5,0);
\draw[<->, gray, opacity=.5, line width=0.5mm] (0,-3.5)--(0,0.5);
\draw [blue, fill=blue] (0,0) circle [radius=.15];
\node at (.8, 0) {$b$};
\node at (0, .8) {$c$};
\node at (-1.5, -4) {$\bd = (1,5,7)$};
\end{tikzpicture}
\end{minipage}
\end{center}
\noindent
Write $(a, b,c)$ for a point in $\R^{Q_0}$. The cone $\sstCone{\bd}$ is drawn in yellow and the walls are drawn in blue for various choices of $\bd$, after projection along the $a$-axis. Note that for the last three choices of $\bd$, every point of $\sstCone{\bd}$ is in a wall.
\end{example}

\begin{example}\label{ex:non-flag}
It is possible that the walls are proper subsets of $\sstCone{\bd}$,
even for examples that are not toric or quiver flag varieties.
One such $(Q, \bd)$ is pictured below left (but see also \cref{ex:mutation}).

\noindent
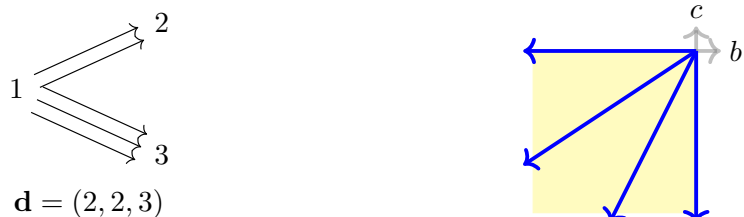
\begin{minipage}{.49\textwidth}
\centering
\begin{tikzcd}[row sep=small, column sep = large]
&2\\
1 \arrow[ur,shift left] \arrow[ur, shift right] \arrow[dr] \arrow[dr, shift left=2] \arrow[dr, shift right=2] \\
& 3
\end{tikzcd}
\[\bd =  (2,2,3)\]
\end{minipage}%
\begin{minipage}{0.49\textwidth}
    \centering
        \begin{tikzpicture}[scale=.65]
\draw[yellow, fill=yellow, opacity=.25] (-3.3,-3.3) rectangle (0, 0);
\draw[<->, gray, opacity=.5, line width=0.5mm] (-3.5,0)--(0.5,0);
\draw[<->, gray, opacity=.5, line width=0.5mm] (0,-3.5)--(0,0.5);
\draw[->, blue, line width=0.5mm] (0,0)--(0,-3.5);
\draw[<-, blue, line width=0.5mm] (-3.5,0)--(0,0);
\draw [->, blue, line width=0.5mm] (0,0)--(-3.5, -2.33);
\draw [->, blue, line width=0.5mm] (0,0)--(-1.75, -3.5);
\node at (.8, 0) {$b$};
\node at (0, .8) {$c$};
\end{tikzpicture}
\end{minipage}

\noindent
Write $(a, b, c)$ for a point in $\R^{Q_0}$. Then $\stabSpace{\mathbf{\bd}}$ is the locus where $2a+2b+3c=0$. We draw this 2-dimensional real vector space at right after projection along the $a$-axis. The yellow region is $\sstCone{\mathbf{d}}$ (it is one quadrant) and the blue lines are the walls $W_\be$: they are the rays through $(-1, 0), (-3, -2), (-1,-2)$, and $(0, -1)$.
\end{example}

\begin{example}\label{ex:mutation}
The behavior of the GIT fan under quiver mutation is relevant for Seiberg duality
(see \cite{Zh:22, Zh:25}, and e.g. \cite[Definition~1.1]{Zh:22} for the definition of mutation for a quiver with dimension vector).
The sequence of quivers and dimension vectors below was obtained by mutating at the vertex labeled 3 in each case.
The arrows in each quiver are labeled with their multiplicity.
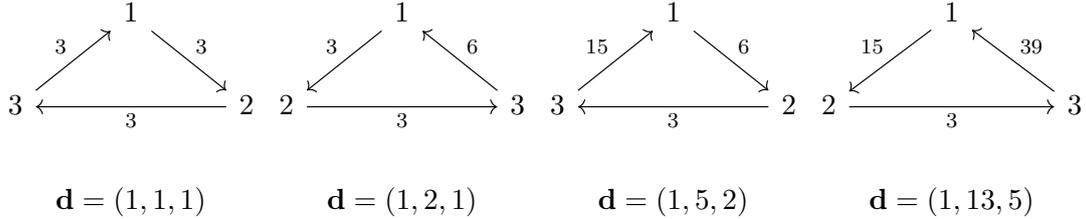
\[
\begin{tikzcd}[column sep = 3pt]
& 1 \arrow[dr, "3"] \\
3  \arrow[ur,"3"] && 2 \arrow[ll, "3"]\\
&\bd = (1,1,1)
\end{tikzcd}
\begin{tikzcd}[column sep = 3pt]
& 1 \arrow[dl, "3"'] \\
2  \arrow[rr,"3"'] && 3 \arrow[ul, "6"']\\
&\bd = (1,2,1)
\end{tikzcd}
\begin{tikzcd}[column sep = 3pt]
& 1 \arrow[dr, "6"] \\
3  \arrow[ur,"15"] && 2 \arrow[ll, "3"]\\
&\bd = (1,5,2)
\end{tikzcd}
\begin{tikzcd}[column sep = 3pt]
& 1 \arrow[dl, "15"'] \\
2  \arrow[rr,"3"'] && 3 \arrow[ul, "39"']\\
&\bd = (1,13,5)
\end{tikzcd}
\]
In each case, we write $(a, b, c)$ for a point in $\mathbb{R}^{Q_0}$
and draw the walls in $\sstCone{\bd}$ after projection along the $a$-axis.
The respective wall diagrams are as follows.

\noindent
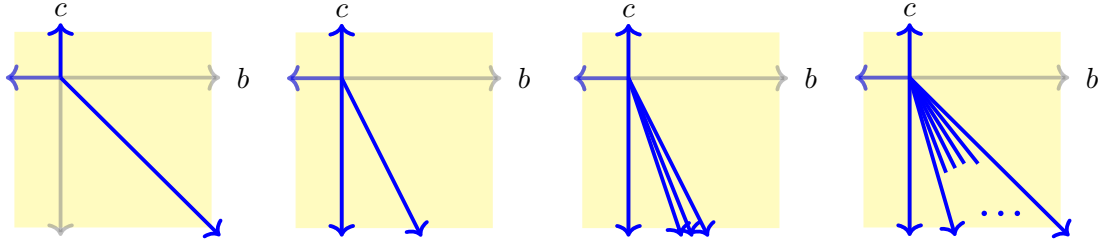
\begin{minipage}{0.24\textwidth}
\centering
\begin{tikzpicture}[scale=.6]
\draw[yellow, fill=yellow, opacity=.25] (-1,-3.3) rectangle (3.3, 1);
\draw[<->, gray, opacity=.5, line width=0.5mm] (-1.2,0)--(3.5,0);
\draw[<->, gray, opacity=.5, line width=0.5mm] (0,-3.5)--(0,1.2);
\draw[<-, blue, opacity=.5, line width=0.5mm] (-1.2, 0)--(0,0);
\draw[<-, blue, line width=0.5mm] (0,1.2)--(0,0);
\draw[->, blue, line width=0.5mm] (0,0)--(3.5,-3.5);
\node at (4, 0) {$b$};
\node at (0, 1.5) {$c$};
\end{tikzpicture}
\end{minipage}
\begin{minipage}{0.25\textwidth}
\centering
\begin{tikzpicture}[scale=.6]
\draw[yellow, fill=yellow, opacity=.25] (-1,-3.3) rectangle (3.3, 1);
\draw[<->, gray, opacity=.5, line width=0.5mm] (-1.2,0)--(3.5,0);
\draw[<->, gray, opacity=.5, line width=0.5mm] (0,-3.5)--(0,1.2);
\draw[<-, blue, opacity=.5, line width=0.5mm] (-1.2, 0)--(0,0);
\draw[<->, blue, line width=0.5mm] (0,1.2)--(0,-3.5);
\draw[->, blue, line width=0.5mm] (0,0)--(1.75,-3.5);
\node at (4, 0) {$b$};
\node at (0, 1.5) {$c$};
\end{tikzpicture}
\end{minipage}
\begin{minipage}{0.25\textwidth}
\centering
\begin{tikzpicture}[scale=.6]
\draw[yellow, fill=yellow, opacity=.25] (-1,-3.3) rectangle (3.3, 1);
\draw[<->, gray, opacity=.5, line width=0.5mm] (-1.2,0)--(3.5,0);
\draw[<->, gray, opacity=.5, line width=0.5mm] (0,-3.5)--(0,1.2);
\draw[<-, blue, opacity=.5, line width=0.5mm] (-1.2, 0)--(0,0);
\draw[<->, blue, line width=0.5mm] (0,1.2)--(0,-3.5);
\draw[->, blue, line width=0.5mm] (0,0)--(1.75,-3.5);
\draw[->, blue, line width=0.5mm] (0,0)--(1.167,-3.5);
\draw[->, blue, line width=0.5mm] (0,0)--(1.4,-3.5);
\node at (4, 0) {$b$};
\node at (0, 1.5) {$c$};
\end{tikzpicture}
\end{minipage}
\begin{minipage}{0.24\textwidth}
\centering
\begin{tikzpicture}[scale=.6]
\draw[yellow, fill=yellow, opacity=.25] (-1,-3.3) rectangle (3.3, 1);
\draw[<->, gray, opacity=.5, line width=0.5mm] (-1.2,0)--(3.5,0);
\draw[<->, gray, opacity=.5, line width=0.5mm] (0,-3.5)--(0,1.2);
\draw[<-, blue, opacity=.5, line width=0.5mm] (-1.2, 0)--(0,0);
\draw[<->, blue, line width=0.5mm] (0,1.2)--(0,-3.5);
\draw[->, blue, line width=0.5mm] (0,0)--(3.5,-3.5);
\draw [-, blue, line width=0.5mm] (0,0)--(1,-2);
\draw [-, blue, line width=0.5mm] (0,0)--(1.2,-1.9);
\draw [-, blue, line width=0.5mm] (0,0)--(.8,-2.1);
\draw [-, blue, line width=0.5mm] (0,0)--(1.5,-1.9);
\draw[->, blue, line width=0.5mm] (0,0)--(1,-3.5);
\node at (4, 0) {$b$};
\node at (2, -3) {\color{blue}\Huge ... };
\node at (0, 1.5) {$c$};
\end{tikzpicture}
\end{minipage}

\noindent
From left to right, the walls are the rays through the following points:
\[\begin{array}{l}
\{(0,1),(-1,0),(1,-1)\}\\
\{(0,1),(-1,0),(0,-1),(1,-2)\}\\
\{(0,1),(-1,0),(0,-1),(1,-3),(1,-2),(2,-5)\}\\
\{(0,1),(-1,0),(0,-1),(1,-3),(1,-2),(1,-1),(2,-5),(3,-7),(3,-8),(5,-13)\}.
\end{array}\]
\end{example}

\begin{example}\label{ex:nonquiver}
For a quiver and dimension vector $(Q, \bd)$,
either the walls do not span $\sstCone{\bd}$,
or the walls contain all of $\sstCone{\bd}$.
The analogous statement is not true for the Geometric Invariant Theory (GIT) fan
determined by an arbitrary representation of a general linear group.
For instance, take \(G \colonequals GL(2)\) and $V \colonequals \mathbb{C}[x,y]_3 \oplus M_{2 \times 4}$,
where $\mathbb{C}[x,y]_3$ is the space of homogeneous polynomials
of degree $3$ and $M_{2 \times 4}$ is the space of $2 \times 4$ matrices.
Let $G$ act on $\mathbb{C}[x,y]_3$ via its natural action on the space of functions
on $\mathbb{A}^2$ and let $G$ act on $M_{2 \times 4}$ by left multiplication.
The additional choice of a character $\theta$ of $GL(2)$
determines open loci $V^{\stable{\theta}} \subseteq V^{\semistable{\theta}} \subseteq V$
of stable and semistable points, respectively (see e.g.~\cite[Definition~2.1]{King:94}).
The GIT fan for $(G, V)$ is pictured below: it is a fan in the character lattice of $GL(2)$,
tensored with $\mathbb{R}$,
with one cone for each of the possible values of $V^{\semistable{\theta}}$.
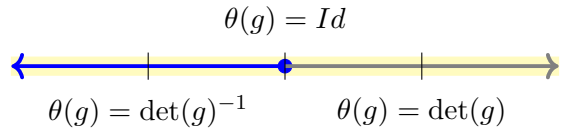
\[\begin{tikzpicture}[scale=.6]
      \draw[yellow, fill=yellow, opacity=.25] (-6,-0.2) rectangle (6, 0.2);
\draw[<->, gray, opacity=.35, line width=0.5mm] (-6,0)--(6,0);
\draw [blue, fill=blue] (0,0) circle [radius=.15];
\draw [<-, blue, line width=0.5mm] (-6, 0)--(0,0);
\draw [->, gray, line width=0.5mm] (0, 0)--(6,0);
\draw [-] (-3,-.3)--(-3, .3);
\draw [-] (3,-.3)--(3, .3);
\draw [-] (0,-.3)--(0, .3);
\node at (-3,-1) {$\theta(g) = \det(g)^{-1}$};
\node at (3,-1) {$\theta(g) = \det(g)$};
\node at (0, 1) {$\theta(g) = Id$};
\end{tikzpicture}
\]
In this case there are three cones:
a $0$- and $1$-dimensional cones where~$V^{\semistable{\theta}}\neq\,V^{\stable{\theta}}$
(the ``walls'', colored blue), and a $1$-dimensional cone
where $V^{\semistable{\theta}} \neq V^{\stable{\theta}}$.
In particular the support of the GIT cone of this representation
is the whole line (shaded yellow)
and the walls are a proper subset whose span is the whole semistable cone.
It follows from the equivalence of (i) and (iii) in \cref{T:phase}
that this fan could never be the GIT fan for a quiver representation.
\end{example}

\begin{example}
\label{ex:segre-cubic}
For the 6-subspaces quiver and the dimension vector $\bd$ pictured below,
it is known (see \cite[p.994]{MR4352662}) that when choosing
the canonical stability parameter $\theta_{\bd}~=~(2,2,2,2,2,2,-6)$,
the resulting quiver moduli is the Segre cubic.

\vspace{0.2cm}
\noindent
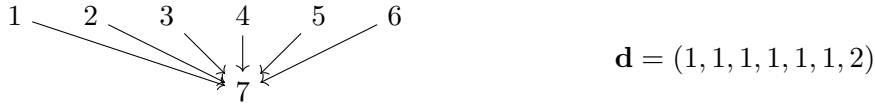
\begin{minipage}{0.5\textwidth}
\centering
\begin{tikzpicture}
\node (1)                 {1};
\node (2) [right of = 1]  {2};
\node (3) [right of = 2]  {3};
\node (4) [right of = 3]  {4};
\node (5) [right of = 4]  {5};
\node (6) [right of = 5]  {6};
\node (7) [below of = 4]  {7};

\draw[->] (1) edge (7);
\draw[->] (2) edge (7);
\draw[->] (3) edge (7);
\draw[->] (4) edge (7);
\draw[->] (5) edge (7);
\draw[->] (6) edge (7);
\end{tikzpicture}
\end{minipage}
\begin{minipage}{0.45\textwidth}
\[\bd = (1, 1, 1, 1, 1, 1, 2)\]
\end{minipage}
\vspace{0.2cm}

\noindent
The support of the GIT fan is the polyhedral cone $\sstCone{\bd}$ of dimension $6$, with $12$ external facets and $15$ rays.
This support is partitioned into a polyhedral fan
with $1678$ $6$-dimensional cones
by $25$ walls of codimension $1$,
and $\theta_{can}$ lays on $10$ of them.
The entire GIT fan has $f_i$ cones of dimension $i - 1$,
for
\begin{equation}
f = (1, 142, 1455, 5200, 8310, 6102, 1678).
\end{equation}
Computations were performed using \software{QuiverTools} and
the \software{Oscar} algebra system,
see~\cite{walls-and-chambers-implementation}.
\end{example}

\printbibliography

\vfill

\emph{Hans Franzen}, \url{hans.franzen.math@gmail.com} \\

\emph{Gianni Petrella}, \url{gianni.petrella@uni.lu} \\
Department of Mathematics, Universit\'{e} de Luxembourg, 2, Avenue de l'Universit\'{e}, 4365 Esch-sur-Alzette, Luxembourg \\

\emph{Rachel Webb}, \url{r.webb@cornell.edu} \\
Department of Mathematics, Cornell University, 310 Malott Hall, Ithaca, NY 14853, United States
\end{document}